\documentclass[12pt, reqno, a4paper]{amsart}

\usepackage{amssymb}
\usepackage{hyperref}
\usepackage[matrix, arrow, curve]{xy}

\usepackage[english]{babel}
\usepackage{ulem}
\usepackage{color}
\theoremstyle{plain}
\newtheorem{theorem}{Theorem}[section]

\newtheorem{proposition}[theorem]{Proposition}
\newtheorem{corollary}[theorem]{Corollary}

\theoremstyle{remark}
\newtheorem{remark}[theorem]{Remark}

\theoremstyle{definition}
\newtheorem{example}[theorem]{Example}

\newtheorem{definition}[theorem]{Definition}

\newtheorem *{Theorem A}{Theorem A}
\newtheorem *{Theorem B}{Theorem B}

\numberwithin{equation}{section}

\DeclareMathOperator{\id}{id}

\DeclareMathOperator{\End}{End}

\DeclareMathOperator{\supp}{supp}

\DeclareMathOperator{\Hom}{Hom}

\begin{document}

\title[Categories and weak equivalences]{Categories and weak equivalences of graded algebras}

\author{Alexey Gordienko}
\address{Vrije Universiteit Brussel, Belgium}
\email{alexey.gordienko@vub.ac.be}

\author{Ofir Schnabel}
\address{Department of Mathematics, Technion - Israel Institute of Technology, Haifa, Israel}
\email{os2519@yahoo.com}

\keywords{Algebra, grading, adjoint functor, oplax $2$-functor, (co)limit, (co)product, (co)equalizer.}

\begin{abstract}
When one studies the structure of graded algebras (e.g. graded ideals, graded subspaces, radicals,\dots) or graded polynomial identities, the grading group itself does not play an important role, but can be replaced by any other group that realizes the same grading. Here we come to the notion of weak equivalence of gradings: two gradings are weakly equivalent if there exists an isomorphism between the graded algebras that maps each graded component onto a graded component. Each group grading on an algebra can be weakly equivalent to $G$-gradings for many different groups $G$, however it turns out that there is one distinguished group among them called the universal group of the grading.
In this paper we study categories and functors related to the notion of weak equivalence of gradings.
In particular, we introduce an oplax 2-functor that assigns to each grading its support and show that the universal grading group functor has neither left nor right adjoint.
\end{abstract}

\subjclass[2010]{Primary 16W50; Secondary 18A20, 18A30, 18A40, 18D20.}

\thanks{The first author is supported by Fonds Wetenschappelijk Onderzoek~--- Vlaanderen post doctoral fellowship (Belgium). The second author was supported by the Israel Science Foundation (grant No. 1516/16).}

\maketitle

\section{Introduction}

When studying graded algebras, one encounters
three main types of equivalences of gradings, namely, isomorphism (see Definition~\ref{DefGradedIsomorphism} below), equivalence (see Definition~\ref{DefGinosarEquivalent}) and weak equivalence (see Definition~\ref{def:weakly}).
 In this paper, we develop a categorical
framework to deal with graded algebras up to a weak equivalence, which
is the least restrictive equivalence among the three equivalences
mentioned above.

Recall that a decomposition $\Gamma \colon A=\bigoplus_{g \in G} A^{(g)}$
of an algebra $A$ over a field $F$ into a direct sum of subspaces $A^{(g)}$
is a \textit{grading} on $A$ by a group $G$
if $A^{(g)}A^{(h)}\subseteq A^{(gh)}$ for all $g,h\in G$.
Then we say that $G$ is the \textit{grading group} of $\Gamma$ and the algebra $A$ is \textit{graded} by $G$. The subspaces $A^{(g)}$ are called \textit{homogeneous} or \textit{graded} components
of $A$ and the elements of $\bigcup_{g\in G} A^{(g)}$ are called \textit{homogeneous} or \textit{graded}.

Let \begin{equation}\label{EqTwoGroupGradings}\Gamma_1 \colon A=\bigoplus_{g \in G} A^{(g)},\qquad \Gamma_2
\colon B=\bigoplus_{h \in H} B^{(h)}\end{equation} be two gradings where $G$
and $H$ are groups and $A$ and $B$ are algebras.

Now we are ready to introduce three types of equivalences of gradings.
The most restrictive case is when we require that both grading groups
coincide:

\begin{definition}[{e.g. \cite[Definition~1.15]{ElduqueKochetov}}]
\label{DefGradedIsomorphism}
The gradings~(\ref{EqTwoGroupGradings}) are \textit{isomorphic} if $G=H$ and there exists an isomorphism $\varphi \colon A \mathrel{\widetilde\to} B$
of algebras such that $\varphi(A^{(g)})=B^{(g)}$
for all $g\in G$.
In this case we say that $A$ and $B$ are \textit{graded isomorphic}.
\end{definition}

In some cases, such as in~\cite{ginosargradings}, less restrictive requirements are more
suitable.
\begin{definition}[{\cite[Definition~2.3]{ginosargradings}}]\label{DefGinosarEquivalent}
The gradings~(\ref{EqTwoGroupGradings})  are \textit{equivalent} if there exists an isomorphism
$\varphi \colon A \mathrel{\widetilde\to} B$
of algebras and an isomorphism $\psi \colon G \mathrel{\widetilde\to} H$
of groups such that $\varphi(A^{(g)})=B^{\bigl(\psi(g)\bigr)}$
for all $g\in G$.
\end{definition}
\begin{remark}
The notion of graded equivalence was considered by Yu.\,A.~Bahturin, S.\,K.~Seghal, and M.\,V.~Zaicev in~\cite[Remark after Definition 3]{BahturinZaicevSeghalGroupGrAssoc}.
In the paper of V.~Mazorchuk and K.~Zhao~\cite{Mazorchuk} it appears under the name of graded isomorphism.
A.~Elduque and M.\,V.~Kochetov refer to this notion as a weak isomorphism
of gradings~\cite[Section~3.1]{ElduqueKochetov}.
More on differences in the terminology in
graded algebras can be found in \cite[\S 2.7]{ginosargradings}.
\end{remark}

 If one studies the graded structure of a graded algebra or its graded polynomial identities~\cite{AljaGia, AljaGiaLa,
BahtZaiGradedExp, GiaLa, ASGordienko9}, then it is not material by elements of which group the graded components are indexed.
A replacement of the grading group leaves both graded subspaces and graded ideals graded.
In the case of graded polynomial identities reindexing
the graded components leads only to renaming the variables.
 Here we come naturally to the notion of weak equivalence of gradings.

\begin{definition}\label{def:weakly}
The gradings~(\ref{EqTwoGroupGradings}) are \textit{weakly equivalent}, if there
exists an isomorphism $\varphi \colon A \mathrel{\widetilde\to}B$
of algebras such that for every $g\in G$ with $A^{(g)}\ne 0$ there
exists $h\in H$ such that $\varphi\left(A^{(g)}\right)=B^{(h)}$.
\end{definition}
 \begin{remark}
 This notion appears in~\cite[Definition~1.14]{ElduqueKochetov}
 under the name of equivalence.
 We have decided to add here the adjective ``weak'' in order to avoid confusion
 with Definition~\ref{DefGinosarEquivalent}.
 \end{remark}

 \begin{example}\label{ExampleM2gradingEquiv}
 Let $A=M_2(F)$, the full $2\times 2$ matrix algebra over a field $F$.
 Denote by $e_{ij}$ the matrix having $1$ in the cell $(i,j)$
 and $0$ in the other cells.
 Then $\Gamma_1 \colon A=\bigoplus_{n \in \mathbb Z} A^{(n)}$
 where $A^{(-1)}=F e_{21}$, $A^{(1)}=F e_{12}$,
 $A^{(0)}=Fe_{11}\oplus Fe_{22}$, $A^{(n)}=0$ for $|n| > 1$,
 is a $\mathbb Z$-grading on $A$.
 At the same time, $\Gamma_2 \colon A=\bigoplus_{\sigma \in S_3} A^{(\sigma)}$
 where $A^{\bigl((123)\bigr)}=F e_{21}$, $A^{\bigl((132)\bigr)}=F e_{12}$,
 $A^{(e)}=Fe_{11}\oplus Fe_{22}$, $A^{(\sigma)}=0$ for $\sigma \ne e, (123),(132)$
 (here $S_3$ is the symmetric group on three elements and $(123),(132)$ are the cycles of length $3$),
 is an $S_3$-grading on $A$ weakly equivalent to $\Gamma_1$ via $\varphi = \id_A$.
 \end{example}

For a grading $\Gamma \colon A=\bigoplus_{g \in G} A^{(g)}$, we denote by $\supp \Gamma := \lbrace g\in G \mid A^{(g)}\ne 0\rbrace$ its support.
\begin{remark}Each weak equivalence $\varphi$ between gradings $\Gamma_1$ and $\Gamma_2$
 induces a bijection $\psi \colon \supp \Gamma_1 \mathrel{\widetilde\to} \supp \Gamma_2$ defined by $\varphi\left( A^{(g)}\right) = B^{\left(\psi(g)\right)}$ for $g\in \supp \Gamma_1$.
\end{remark}

Obviously, if gradings are isomorphic, then they are equivalent and if they are equivalent then they are also weakly equivalent.
It is important to notice that none of the converse is true.
However, if gradings~(\ref{EqTwoGroupGradings}) are weakly equivalent
and $\varphi \colon A \mathrel{\widetilde\to}B$ is the corresponding isomorphism of algebras,
then $\Gamma_3 \colon A=\bigoplus_{h \in H} \varphi^{-1}\left( B^{(h)}\right)$ is a grading on $A$ isomorphic to $\Gamma_2$ and the grading $\Gamma_3$ is obtained from $\Gamma_1$ just by reindexing the homogeneous components.
Therefore, when gradings~(\ref{EqTwoGroupGradings}) are weakly
equivalent, we say that $\Gamma_1$ \textit{can be regraded} by $T$.
If $A=B$ and $\varphi$ in Definition~\ref{def:weakly} is the identity map, we say that $\Gamma_1$ and $\Gamma_2$ are realizations
of the same grading on $A$ as, respectively, an $G$- and a $H$-grading.

As we have mentioned above, for many applications it is not important
which particular grading among weakly equivalent ones
we consider. Thus, if it is possible, one can try to replace
the grading group by a better one. For example, in~\cite{ClaseJespersDelRio} and~\cite{GordienkoSchnabel}  the authors studied the question, whether it is always possible to regrade a grading by a finite group. The situation, when this is possible, is very convenient since the algebra graded by a finite group $G$
is an $FG$-comodule algebra and, in turn, an $(FG)^*$-module algebra
where $FG$ is the group algebra of $G$, which is a Hopf algebra, and $(FG)^*$ is its dual.
In this case one can use the techniques of Hopf algebra actions instead of working with a grading directly (see e.g.~\cite{ASGordienko3}).

Each group grading on an algebra can be realized as a $G$-grading for many different groups $G$, however it turns out that there is one distinguished group among them \cite[Definition~1.17]{ElduqueKochetov}, \cite{PZ89}.

\begin{definition}\label{def:universal}
Let $\Gamma$ be a group grading on an algebra $A$. Suppose that $\Gamma$ admits a realization
as a $G_\Gamma$-grading for some group $G_\Gamma$. Denote by $\varkappa_\Gamma$ the corresponding embedding
$\supp \Gamma \hookrightarrow G_\Gamma$. We say that $(G_\Gamma,\varkappa_\Gamma)$ is the \textit{universal group of the grading $\Gamma$} if for any realization of $\Gamma$ as a grading by a group $G$
with $\psi \colon \supp \Gamma \hookrightarrow G$ there exists
a unique homomorphism $\varphi \colon G_\Gamma \to G$ such that the following diagram is commutative:
$$\xymatrix{ \supp \Gamma \ar[r]^\varkappa \ar[rd]^\psi & G_\Gamma \ar@{-->}[d]^\varphi \\
& G
}
$$
\end{definition}

Given a set $X$, denote by $\mathcal F(X)$ the free group with the set $X$ of free generators.
It is easy to see that if $G$ is a group and $\Gamma \colon A = \bigoplus_{g\in G} A^{(g)}$ is a grading, then $$G_\Gamma \cong \mathcal F([\supp \Gamma])/N$$ where $[\supp \Gamma]:=\lbrace [g] \mid g\in \supp \Gamma \rbrace$ and $N$ is the normal closure of the words $[g][h][gh]^{-1}$ for $g,h \in \supp \Gamma$
such that $A^{(g)}A^{(h)}\ne 0$.

\begin{example}
In Example~\ref{ExampleM2gradingEquiv} the universal group of the gradings $\Gamma_1$ and $\Gamma_2$
is isomorphic to $\mathbb Z$ where $\lbrace -1, 0, 1\rbrace \subseteq \mathbb Z$ is the embedding of the support.
\end{example}

In the definition above the universal group of a grading $\Gamma$ is a pair $(G_\Gamma,\varkappa_\Gamma)$.
In~\cite[Theorem~4.3]{GordienkoSchnabel} the authors have shown that the first component of this pair can be an arbitrary finitely presented group. More precisely, for any finitely presented group $G$ there exists
$n\in\mathbb N$ and an elementary grading $\Gamma$ on the algebra of all $n\times n$ matrices  such that $G_\Gamma\cong G$.

\begin{remark} For each grading $\Gamma$ one can define a category $\mathcal C_\Gamma$
where the objects are all pairs $(G,\psi)$ such that $G$ is a group and $\Gamma$ can be realized
as a $G$-grading with $\psi \colon \supp \Gamma \hookrightarrow G$ being the embedding of the support.
In this category the set of morphisms between $(G_1,\psi_1)$ and $(G_2,\psi_2)$ consists
of all group homomorphisms $f \colon G_1 \to G_2$ such that the diagram below is commutative:
$$\xymatrix{ \supp \Gamma \ar[r]^{\psi_1} \ar[rd]^{\psi_2} & G_1 \ar[d]^f \\
& G_2
}
$$
Then $(G_\Gamma,\varkappa_\Gamma)$ is the initial object of $\mathcal C_\Gamma$.
\end{remark}

Categories and functors related to graded algebras have been studied extensively (see e.g. \cite{NastasescuVanOyst}) and several important pairs of adjoint functors have been noticed (see Section~\ref{SectionGradedAdjunctions} for a brief summary).
Each category consisted of algebras graded by a fixed group, i.e. isomorphisms in those categories coincided with isomorphisms of gradings. In order to obtain a proper categorical framework for weak equivalences, we
introduce the category, which we denote by $\mathbf{GrAlg}_F$ (or $\mathbf{GrAlg}^1_F$, if we restrict ourselves to unital algebras and unital homomorphisms), that consists of all associative algebras
graded by different groups where the morphisms are all homomorphisms of algebras that map each graded component into some graded component. Here the induced partial maps
on the grading groups come into play naturally.

In Section~\ref{SectionOplaxSupport}
 we show that the assignment of  the support to a grading leads to an oplax $2$-functor. In order to get an ordinary functor $R$ (respectively, $R_1$,
 if will deal with unital algebras) from the category of graded algebras to the category of groups that assigns
 to each grading its universal group, we restrict the sets of morphisms in the categories of graded algebras to the sets of graded homomorphisms that are injective on each graded component (Section~\ref{SectionUniversalGradingGroupFunctors}). We call such homomorphisms
 \textit{graded injective} and denote the resulting categories by $\widetilde{\mathbf{GrAlg}_F}$ and $\widetilde{\mathbf{GrAlg}^1_F}$.

In Sections~\ref{SectionLimColimGrAlg} and~\ref{SectionLimColimTildeGrAlg}
we study in detail existence of limits and colimits in $\mathbf{GrAlg}_F$, $\mathbf{GrAlg}^1_F$, $\widetilde{\mathbf{GrAlg}_F}$, and $\widetilde{\mathbf{GrAlg}^1_F}$.
We give examples of graded algebras for which (co)products and coequalizers do not exist and, in particular, we prove
\begin{Theorem A}
Let $A$ and $B$ be objects in $\mathbf{GrAlg}_F$ or $\mathbf{GrAlg}^1_F$.
A product of $A$ and $B$ in $\mathbf{GrAlg}_F$ or, respectively, $\mathbf{GrAlg}^1_F$ exists if and only if either of $A$ and $B$ is
the zero algebra or both $A$ and $B$ are algebras with a trivial grading.
\end{Theorem A}

We also calculate the product of two group algebras in $\widetilde{\mathbf{GrAlg}_F}$ and $\widetilde{\mathbf{GrAlg}^1_F}$ (Proposition~\ref{PropositionProductGroupAlgebrasTildeGrAlg}),
describe equalizers (Proposition~\ref{PropositionEqualizerGrAlg}) and monomorphisms (Propositions~\ref{PropositionCriterionForMonomorphismGrAlg} and~\ref{PropositionCriterionForMonomorphismTildeGrAlg}).

In Section~\ref{SectionAbsenceGrAdjunctions} we prove
 \begin{Theorem B}
The universal grading group functors
$$R\colon \widetilde{\mathbf{GrAlg}_F} \rightarrow \mathbf{Grp} \text{ and }
R_1\colon \widetilde{\mathbf{GrAlg}^1_F}\rightarrow \mathbf{Grp}$$
admit neither left nor right adjoints.
\end{Theorem B}

It turns out that in order to force $R$ and $R_1$ to have a left adjoint we have to restrict our consideration very much, e.g. to group algebras of groups that do not have non-trivial one dimensional representations. In that case we even get an isomorphism of categories (Proposition~\ref{PropositionEquivGroupAlgGroups}).

To sum up, even without a restriction of the sets of morphisms,
the categories $\mathbf{GrAlg}_F$ and $\mathbf{GrAlg}^1_F$ of algebras graded by all groups are quite different from the categories where the grading group is fixed.

\section{Change of the grading group and the free-forgetful adjunction}
\label{SectionGradedAdjunctions}

In order to give the reader an opportunity to compare the categories and functors
constructed and studied in the current paper with those studied in the literature earlier,
we recall two classical examples of adjunctions in the categories related to graded algebras.
The algebras below are associative, but not necessarily unital or commutative.
Our particular interest in non-unital algebras originates from the theory of polynomial
identities where the adjunction of the unit element to an algebra can completely change its polynomial
identities.

If $A$ and $B$ are objects in a category $\mathcal A$, we denote the set of morphisms $A\to B$
by $\mathcal A(A,B)$.

Let $\mathbf{Alg}_F^G$ be the category of (not necessarily unital)  associative algebras over a field $F$ graded by a group $G$. The morphisms in $\mathbf{Alg}_F^G$ are all the homomorphisms of algebras
$$\psi \colon A=\bigoplus_{g\in G}A^{(g)}\ \longrightarrow\ B=\bigoplus_{g\in G}B^{(g)}$$
such that $\psi(A^{(g)})\subseteq B^{(g)}$ for all $g\in G$.
For any group homomorphism $\varphi \colon G \to H$,
denote by $U_\varphi$  the functor $ \mathbf{Alg}_F^G \to \mathbf{Alg}_F^H$ that assigns to any $G$-grading $A=\bigoplus_{g\in G}A^{(g)}$ on an algebra $A$ the $H$-grading $$A = \bigoplus_{h\in H} A^{(h)}\text{ where }A^{(h)} := \bigoplus_{\substack{g\in G,\\ \varphi(g)=h}} A^{(g)}$$
and does not change homomorphisms.

The functor $U_\varphi  \colon \mathbf{Alg}_F^G \to \mathbf{Alg}_F^H$ admits the right adjoint functor $K_\varphi \colon \mathbf{Alg}_F^H \to \mathbf{Alg}_F^G$ defined as follows:
for an $H$-grading $B=\bigoplus_{h\in H} B^{(h)}$ we have $K_\varphi (B) = \bigoplus_{g\in G} \bigl(K_\varphi(B)\bigr)^{(g)}$
where $\bigl(K_\varphi(B)\bigr)^{(g)} := \lbrace (g, b) \mid b\in B^{\left(\varphi(g)\right)} \rbrace$.
The vector space structure on $\bigl(K_\varphi(B)\bigr)^{(g)}$ is induced from that on $B^{\left(\varphi(g)\right)}$ and
the multiplication in $K_\varphi(B)$ is defined by $(g_1,a)(g_2,b):=(g_1 g_2, ab)$ for $g_1, g_2\in G$, $a\in B^{\left(\varphi(g_1)\right)}$, $b\in B^{\left(\varphi(g_2)\right)}$. For $\psi \in \mathbf{Alg}_F^H(B_1,B_2)$,
the morphism $K_\varphi (\psi)\in \mathbf{Alg}_F^G$ is defined by $K_\varphi (\psi)(g,b):=(g,\psi(b))$ where $g\in G$, $b\in B_1^{\left(\varphi(g)\right)}$.
We have a natural bijection
$$\mathbf{Alg}_F^H (U_\varphi (A), B)\to \mathbf{Alg}_F^G (A, K_\varphi(B))$$
where $A \in \textbf{Alg}_F^G$, $B \in \textbf{Alg}_F^H$ (see e.g.~\cite[Section~1.2]{NastasescuVanOyst}).

Another example of an adjunction is the free-forgetful one. It is especially important for the theory of graded polynomial identities~\cite{AljaGia, AljaGiaLa,
BahtZaiGradedExp, GiaLa, ASGordienko9}.

Let $G$ be a group and let $\mathbf{Set}^G_*$ be the category whose objects are sets
$X$, containing a distinguished element $0$, with a fixed decomposition $X=\lbrace 0 \rbrace \sqcup \bigsqcup_{g\in G} X^{(g)}$ into a disjoint union.
Morphisms between $X=\lbrace 0 \rbrace \sqcup \bigsqcup_{g\in G} X^{(g)}$
and $Y=\lbrace 0 \rbrace \sqcup \bigsqcup_{g\in G} Y^{(g)}$ are maps $\varphi \colon X \to Y$
such that $\varphi(0)=0$ and $\varphi(X^{(g)})\subseteq \lbrace 0 \rbrace \sqcup Y^{(g)}$
for all $g\in G$. We have an obvious forgetful functor $U \colon \textbf{Alg}_F^G \to \mathbf{Set}^G_*$
that assigns to each graded algebra $A$ the object $U(A)=\lbrace 0 \rbrace \sqcup \bigsqcup_{g\in G} \left(A^{(g)} \backslash \lbrace 0 \rbrace \right)$.
This functor has a left adjoint functor $F\langle (-)\backslash \lbrace 0 \rbrace  \rangle$
that assigns to $X=\lbrace 0 \rbrace \sqcup \bigsqcup_{g\in G} X^{(g)} \in \mathbf{Set}^G_*$
the free non-unital associative algebra $F\langle X\backslash \lbrace 0 \rbrace  \rangle$ on the set $X\backslash \lbrace 0 \rbrace$ which is endowed with the grading defined by $x_1 \cdots x_n \in
F\langle X\backslash \lbrace 0 \rbrace  \rangle^{(g_1 \cdots g_n)}$
for $x_i \in X^{(g_i)}$, $1\leqslant i \leqslant n$.
Here we have a natural bijection
$$\mathbf{Alg}_F^G (F\langle X\backslash \lbrace 0 \rbrace\rangle, A)\to \mathbf{Set}^G_* (X, U(A))$$
where $A \in \textbf{Alg}_F^G$, $X \in \mathbf{Set}^G_*$.

First of all, we notice that both examples deal with the categories $\mathbf{Alg}_F^G$ where for each category the grading group is fixed, i.e. the notion of isomorphism in $\mathbf{Alg}_F^G$ coincides with the notion of graded isomorphism. Second, although the functor $U_\varphi $ changes the grading group,
 the grading on $U_\varphi(A)$ is not necessarily weakly equivalent to the grading on $A$.

In order to deal with weak equivalences of gradings, below we introduce the category $\mathbf{GrAlg}_F$ of algebras over the field $F$ graded by arbitrary groups, in which the notion of isomorphism coincides with the notion of weak equivalence of gradings.

\section{The category $\mathbf{GrAlg}_F$ and the corresponding oplax $2$-functor}
\label{SectionOplaxSupport}

We say that a homomorphism $\psi \colon A \to B$
between graded algebras $A=\bigoplus_{g \in G} A^{(g)}$
and $B=\bigoplus_{h \in H} B^{(h)}$ is \textit{graded}
if for every $g\in G$ there exists $h\in H$ such that $\psi(A^{(g)})\subseteq B^{(h)}$.

Any graded homomorphism of graded algebras induces a map between subsets of the supports of the gradings.
This gives rise to several functors which we study below.

Let $\mathbf{GrAlg}_F$ be the category where the objects are group gradings on (not necessarily unital)  associative algebras over a field $F$ and morphisms are graded homomorphisms between the corresponding algebras.

By $\mathbf{GrAlg}^1_F$ we denote the category where the objects are group gradings on unital associative algebras over a field $F$ and morphisms are graded unital homomorphisms between the corresponding algebras.

Consider also the category $\mathcal C$ where:
\begin{itemize}
\item the objects are triples $(G, S, P)$ where $G$ is a group, $S \subseteq G$ is a subset, and $P \subseteq S\times S$ is a subset such that $\lbrace gh \mid (g,h)\in P \rbrace \subseteq S$;
\item the morphisms  $(G_1, S_1, P_1)\to (G_2, S_2, P_2)$ are triples $(\psi, R, Q)$
where $R \subseteq S_1$, $Q\subseteq P_1 \cap (R\times R)$, and $\psi \colon R \to S_2$ is a map
such that $\lbrace gh \mid (g,h)\in Q\rbrace
\subseteq R$, $\psi(g)\psi(h)=\psi(gh)$ for all $(g,h)\in Q$ and $\lbrace(\psi(g),\psi(h))
\mid (g,h)\in Q \rbrace \subseteq P_2$;
\item the identity morphism for $(G, S, P)$ is $(\id_{S}, S, P)$;
\item if $(\psi_1, R_1, Q_1) \colon (G_1, S_1, P_1)\to (G_2, S_2, P_2)$
and $(\psi_2, R_2, Q_2) \colon (G_2, S_2, P_2)\to (G_3, S_3, P_3)$,
then $$(\psi,R,Q)=(\psi_2, R_2, Q_2)(\psi_1, R_1, Q_1) \colon (G_1, S_1, P_1)\to (G_3, S_3, P_3)$$
is defined by $R=\lbrace g\in R_1 \mid \psi_1(g)\in R_2 \rbrace$, $Q = \lbrace (g,h)\in Q_1 \mid (\psi_1(g),\psi_1(h))\in Q_2 \rbrace$
and $\psi(g)=\psi_2(\psi_1(g))$ for $g\in R$.
\end{itemize}

There exists an obvious functor-like map $L \colon \mathbf{GrAlg}_F \to \mathcal C$ where for $\Gamma \colon A=\bigoplus_{g\in G} A^{(g)}$ we have $$L(\Gamma):=(G, \supp \Gamma, \lbrace
(g_1,g_2)\in G \times G \mid A^{(g_1)}A^{(g_2)}\ne 0\rbrace)$$
and for a graded morphism $\varphi \colon \Gamma \to
\Gamma_1$,
where $\Gamma_1 \colon B=\bigoplus_{h\in H} B^{(h)}$,
the triple $L(\varphi)=(\psi, R, Q)$ is defined by $R=\lbrace g\in G \mid \varphi\left(A^{(g)}\right)\ne 0 \rbrace$, $$Q=\lbrace(g_1,g_2)\in R\times R\mid \varphi\left(A^{(g_1)}\right)\varphi\left(A^{(g_2)}\right)\ne 0\rbrace,$$
the map $\psi$ is defined by $\varphi\left(A^{(g)}\right) \subseteq B^{(\psi(g))}$
 for $g\in R$.

 As we will see in Example~\ref{ExampleLnotAFunctor} below,
 $L$ is not an ordinary functor: $L(\varphi_1)L(\varphi_2)$
is not necessarily equal to $L(\varphi_1 \varphi_2)$.
In order to overcome this difficulty, we endow the sets of morphisms
$\mathcal C((G_1, S_1, P_1),(G_2, S_2, P_2))$ with a partial ordering $\preccurlyeq$:
we say that $(\psi_1, R_1, Q_1) \preccurlyeq (\psi_2, R_2, Q_2)$
if $R_1 \subseteq R_2$, $Q_1 \subseteq Q_2$, and $\psi_1 = \psi_2 \bigl|_{R_1}$. For every pair $\varphi_1, \varphi_2$ of composable morphisms in $\mathcal C$ we have \begin{equation}\label{EqLSucc}L(\varphi_1)L(\varphi_2)
\succcurlyeq L(\varphi_1 \varphi_2),\end{equation}
but the inequality~(\ref{EqLSucc}) can be strict:

\begin{example}\label{ExampleLnotAFunctor}
Let $A=A^{(\bar 0)} \oplus A^{(\bar 1)}$ be a $\mathbb Z/2\mathbb Z$-graded algebra,
where $$A^{(\bar 0)}=F1_A,\quad A^{(\bar 1)}=Fa\oplus Fb,\quad a^2=ab=ba=b^2=0.$$
Let $\varphi \colon A \to A$ be the graded homomorphism defined by $\varphi(1_A)=1_A$, $\varphi(a)=b$, $\varphi(b)=0$.
Then $$L(A)=(\mathbb Z/2\mathbb Z, \mathbb Z/2\mathbb Z, (\mathbb Z/2\mathbb Z)^2 \backslash \lbrace (\bar 1, \bar 1)\rbrace),$$ $$L(\varphi)=(\id_{\mathbb Z/2\mathbb Z}, \mathbb Z/2\mathbb Z,
(\mathbb Z/2\mathbb Z)^2 \backslash \lbrace (\bar 1, \bar 1)\rbrace),$$
and $L(\varphi)^2=L(\varphi)$, however $$L(\varphi^2)=(\id_{\lbrace \bar 0 \rbrace}, \lbrace \bar 0 \rbrace, \lbrace (\bar 0, \bar 0) \rbrace)\prec L(\varphi)^2.$$
\end{example}

Now we are going to put $L$ in an appropriate categorical framework using the notion of an enriched category.
A category $\mathcal A$ is
\textit{enriched} over another category $\mathcal B$, if
 the hom-objects $\mathcal A(a,b)$ are objects of $\mathcal B$
and the composition and the assignment of the identity morphism
are morphisms in $\mathcal B$. (See the precise definition in~\cite[Section 1.2]{KellyEnriched}.)

Recall that each partially ordered set (or poset) $(M, \preccurlyeq)$ is a category where the objects are the elements $m\in M$ and if $m \preccurlyeq n$, then there exists a single morphism $m \to n$. If $m \not\preccurlyeq n$, then there is no morphism $m \to n$.
Denote by $\mathbf{Cat}$ the category of small categories. Since the notion of a $\mathbf{Cat}$-enriched category coincides with the notion of a $2$-category, every category enriched over
posets is a $2$-category.
The partial ordering $\preccurlyeq$ turns $\mathcal C$ into a category enriched over posets and, in particular, a $2$-category where
$0$-cells are triples $(G, S, P)$, $1$-cells are triples $(\psi, R, Q)$
and from $1$-cell $(\psi_1, R_1, Q_1)$ to $1$-cell $(\psi_2, R_2, Q_2)$ there exists a $2$-cell if and only if $(\psi_1, R_1, Q_1) \preccurlyeq (\psi_2, R_2, Q_2)$.

 Between $2$-categories one can consider \textit{strict} $2$-functors that preserve the composition laws and the identity $1$- and $2$-cells, as well as \textit{(op)lax} $2$-functors $F$ where equalities $F(\varphi_1)F(\varphi_2)=F(\varphi_1 \varphi_2)$
 for $1$-cells $\varphi_1,\varphi_2$ are substituted with
 an existence of a $2$-cell between the corresponding $1$-cells
 $F(\varphi_1)F(\varphi_2)$ and $F(\varphi_1 \varphi_2)$.
The difference between a lax $2$-functor and an oplax $2$-functor
is in the direction in which that $2$-cell goes.
 (See the details and the precise definition of an (op)lax 2-functor in \cite[p.~83]{KellyTwoCat} and \cite[p.~29]{BenabouBi}.)

In this terminology, the inequality~(\ref{EqLSucc})
means that there is a $2$-cell between $L(\varphi_1)L(\varphi_2)$
and $L(\varphi_1 \varphi_2)$. This turns $L$ to an oplax $2$-functor between $\mathbf{GrAlg}_F$
and $\mathcal C$ if we treat the category $\mathbf{GrAlg}_F$ as a $2$-category with discrete hom-categories (i.e. the only $2$-cells in $\mathbf{GrAlg}_F$ are the identity $2$-cells between morphisms). All the equalities in the definition of an oplax $2$-functor hold, since in the categories that are posets all diagrams commute.

Obviously, one can restrict the domain of $L$ to the category $\mathbf{GrAlg}^1_F$
and consider the same phenomenon in the case of unital graded algebras.

\section{The universal grading group functors}\label{SectionUniversalGradingGroupFunctors}

In the previous section we showed that the functor-like map $L$ that assigns to each grading its
support is not an ordinary functor, but a $2$-functor. In the current section we would like
to introduce a functor that is similar to $L$, but is an ordinary functor.
In order to do so, we restrict the sets of possible morphisms. In addition, instead of considering the support of the grading, we assign to a grading the group generated by the support and the corresponding relations, namely, the universal grading group.

We call a graded homomorphism \textit{graded injective} if its restriction to each homogeneous component
is an injective map. For example, any unital graded homomorphism of a group algebra to any unital graded algebra is graded injective since all group elements are mapped to invertible elements.

Consider the category $\widetilde{\mathbf{GrAlg}_F}$ where the objects are group gradings on (not necessarily unital)  associative algebras over a field $F$ and morphisms are  graded injective homomorphisms between the corresponding algebras. We have an obvious functor $R \colon \widetilde{\mathbf{GrAlg}_F} \to \mathbf{Grp}$
where for $\Gamma \colon A=\bigoplus_{g\in G} A^{(g)}$ we have $R(\Gamma):=G_{\Gamma}$
and for a graded injective morphism $\varphi \colon
\Gamma \to \Gamma_1$, where $\Gamma_1 \colon B=\bigoplus_{h\in H} B^{(h)}$,
the group homomorphism $R(\varphi) \colon G_{\Gamma} \to G_{\Gamma_1}$ is defined by
$R(\varphi)(g):=h$ where $\varphi\left(A^{(g)}\right) \subseteq B^{(h)}$, $g\in G$.

One can also restrict his consideration to unital algebras.
Denote by $\widetilde{\mathbf{GrAlg}_F^1}$  the category where the objects are group gradings on unital  associative algebras over a field $F$ and morphisms are unital graded injective homomorphisms between the corresponding algebras.
Denote by $R_1$ the restriction of the functor $R$ to the category $\widetilde{\mathbf{GrAlg}_F^1}$. We call $R$ and $R_1$ the \textit{universal grading group functors}.

When it is clear, with which grading an algebra $A$ is endowed, we identify $A$ with the grading on it and treat $A$ as an object of the corresponding category.

In the sections below we study the categories $\mathbf{GrAlg}_F$,
$\mathbf{GrAlg}^1_F$, $\widetilde{\mathbf{GrAlg}_F}$,
and $\widetilde{\mathbf{GrAlg}^1_F}$ as well as the functors $R$ and $R_1$ introduced above.

\section{Categorical constructions in $\mathbf{GrAlg}_F$
and $\mathbf{GrAlg}^1_F$}\label{SectionLimColimGrAlg}

In this section we investigate the existence of (co)products and (co)equalizers in
$\mathbf{GrAlg}_F$ and in $\mathbf{GrAlg}^1_F$.
Our particular interest to these specific examples of (co)limits
originates from the well-known fact that if
in a category there exist (co)products and (co)equalizers,
then there exist all (co)limits. (See e.g. \cite[Chapter~V, Section 2, Theorem~2]{MacLaneCatWork}.)
In addition, in Proposition~\ref{PropositionCriterionForMonomorphismGrAlg}
below we describe monomorphisms and compare them with injective and graded injective homomorphisms.

It is easy to see that the product of a graded algebra $A$ with the zero algebra
is isomorphic to $A$. Also, if two algebras $A$ and $B$ both have the trivial grading, then their product
as non-graded algebras (again with the trivial grading) will be still their product in $\mathbf{GrAlg}_F$ and, if $A$ and $B$ are unital, in $\mathbf{GrAlg}^1_F$.
Theorem~\ref{TheoremAbsensceOfProductsGrAlg} below shows that in any other case, a product of two objects in $\mathbf{GrAlg}_F$ and in $\mathbf{GrAlg}^1_F$ does not exists.
As a consequence, we get Theorem~A.

\begin{theorem}\label{TheoremAbsensceOfProductsGrAlg}
Let $\Gamma_1 \colon A = \bigoplus_{g\in G} A^{(g)}$ and $\Gamma_2 \colon B = \bigoplus_{h\in H} B^{(h)}$ be group gradings on algebras over a field $F$ such that
$\supp \Gamma_1$ consist of at least two distinct
elements and $B\ne 0$. Then the product of $A$ and $B$ exists neither
in $\mathbf{GrAlg}_F$ nor in $\mathbf{GrAlg}^1_F$.
\end{theorem}
\begin{proof}
We will prove the both cases $\mathbf{GrAlg}_F$ and
$\mathbf{GrAlg}^1_F$ simultaneously. In the latter case $A$ and $B$ are assumed to be unital.

Choose $g,t\in G$, $h\in H$, $a^{(g)} \in A^{(g)}$, $a^{(t)} \in A^{(t)}$,
$b^{(h)} \in B^{(h)}$ such that $g\ne t$, $a^{(g)}\ne 0$, $a^{(t)}\ne 0$,
$b^{(h)}\ne 0$.

Suppose the product $A \times B$ exists. Let $\pi \colon A\times B \to A$
and $\rho \colon A\times B \to B$ be the corresponding morphisms (``projections'').

Let $D$ be the free associative non-commutative algebra (unital if we work in $\mathbf{GrAlg}^1_F$ and non-unital if we work in $\mathbf{GrAlg}_F$) over a field $F$ in the variables $x, y$ with the $\mathbb Z$-grading by the total degree of a monomial. The definition of a product in a category implies that for every morphisms $\alpha \colon D \to A$ and $\beta \colon D \to B$
there exists a unique morphism $\psi \colon D \to A\times B$
such that the diagram below is commutative:

$$\xymatrix{ & A \times B
\ar[ld]_{\pi} \ar[rd]^{\rho} & \\
A & & B \\
& D \ar[lu]_{\alpha} \ar[ru]^{\beta} \ar@{-->}[uu]^{\psi} &
}
$$

Define graded homomorphisms $\alpha_1, \alpha_2 \colon D \to A$
and $\beta \colon D \to B$ by $\alpha_1(x)=\alpha_2(x)=0$, $\alpha_1(y)=a^{(g)}$,
 $\alpha_2(y)=a^{(t)}$, $\beta(x)=\beta(y)=b^{(h)}$.
 Now let $\psi_1, \psi_2 \colon D \to A\times B$ be the unique graded homomorphisms
 such that $\pi \psi_i = \alpha_i$ and $\rho\psi_i = \beta$.

Denote by $C$ the algebra of all polynomials
in the variable $x$ with coefficients in $F$ without a constant term if we work in $\mathbf{GrAlg}_F$ and with a constant term if we work in $\mathbf{GrAlg}^1_F$. We endow $C$ with the degree $\mathbb Z$-grading. Then each morphism $\alpha \colon C \to A$ or $\beta \colon C \to B$
is determined uniquely by the choice of a homogeneous element $\alpha(x)$ or $\beta(x)$.
Let $\tau \colon C \hookrightarrow D$ be the natural embedding where $\tau(x)=x$.
Since $\pi\psi_1\tau(x)=\alpha_1(x)=\alpha_2(x)=\pi\psi_2\tau(x)$
and $\rho\psi_1\tau(x)=\beta(x)=\rho\psi_2\tau(x)$, the universal property of the product
implies $\psi_1\tau = \psi_2\tau$ and $\psi_1(x)=\psi_2(x)$. At the same time,
as $\rho(\psi_1(x))=b^{(h)}\ne 0$, we have $\psi_1(x)=\psi_2(x)\ne 0$.
The same argument implies $\psi_1(y)\ne 0$ and $\psi_2(y)\ne 0$.
Since $x$ and $y$ belong to the same homogeneous component of $D$,
$\psi_1(x)$ and $\psi_1(y)$ belong to the same homogeneous component of $A\times B$.
Analogously, $\psi_2(x)$ and $\psi_2(y)$ belong to the same homogeneous component of $A\times B$. As all these elements are nonzero
and $\psi_1(x)=\psi_2(x)$, we obtain that $\psi_1(y)$ and $\psi_2(y)$ belong to the same homogeneous component of $A\times B$. In particular, $a^{(g)}=\pi(\psi_1(y))$
and $a^{(t)}=\pi(\psi_2(y))$ belong to the same  homogeneous component of $A$.
Thus $g=t$ and we get a contradiction.
Therefore, $A\times B$ exists neither
in $\mathbf{GrAlg}_F$ nor in $\mathbf{GrAlg}^1_F$.
\end{proof}

Now we give an example of graded algebras for which the coproduct does not exist.

\begin{proposition}
Let $A$ be an associative algebra over a field $F$ with the basis $1_A, a, b, c, d, cd$
where all the products of $a,b,c,d$ are zero except $cd=dc\ne 0$.
Let $B=\langle 1_B, v\rangle_F$, $v^2=0$. Denote by $\xi$ a generator of the cyclic group $C_3$
of order $3$. Consider the $C_3$-grading on $A$
defined by $A^{(1)}=\langle 1_A, cd\rangle_F$, $A^{(\xi)}=\langle a,c\rangle_F$,
$A^{\left(\xi^2\right)}=\langle b,d\rangle_F$ and the trivial grading on $B$.
Then the coproduct of $A$ and $B$ exists neither
in $\mathbf{GrAlg}_F$ nor in $\mathbf{GrAlg}^1_F$.
\end{proposition}
\begin{proof}
Suppose there exists a coproduct which we denote by $A\sqcup B$.
Denote by $i\colon A \to A\sqcup B$ and $j\colon B \to A\sqcup B$
the corresponding morphisms.

Let $f\colon B \to A$ be the unital homomorphism defined by $f(v)=0$.
Then there exists a unique morphism $g\colon A \sqcup B \to A$ such that the diagram below is commutative:

$$\xymatrix{ & A \sqcup B
\ar@{-->}[dd]^h & \\
A \ar[ru]^{i}\ar@{=}[rd]^{\id_A} & & B \ar[lu]_{j}\ar[ld]_{f}
 \\
& A   &
}
$$

Since $cd=h(i(cd))\ne 0$, we have \begin{equation}\label{EqIcIdNe0}
i(c)i(d)=i(d)i(c)\ne 0.
\end{equation}

Consider the associative algebra $C$ with the basis $\lbrace 1_C, x, y, z, xy, yz, zy, yx, xyz, zyx\rbrace$
where all the other products of $x,y,z$ are zero.
Define a grading on $C$ by the free group $\mathcal F (X,Z)$ by $x\in C^{(X)}$; $1_C,y\in C^{(1)}$; $z \in C^{(Z)}$. Define unital graded homomorphisms $\alpha \colon A \to C$
and $\beta \colon B \to C$ by
$$\alpha(a)=x,\text{ } \alpha(b)=z,\text{ } \alpha(c)=0, \text{ } \alpha(d)=0, \text{ }\beta(v)=y.$$

Then there exists a unique morphism $\psi \colon A\sqcup B \to C$ such that the diagram below is commutative:

$$\xymatrix{ & A \sqcup B
\ar@{-->}[dd]^\psi & \\
A \ar[ru]^{i}\ar[rd]^{\alpha} & & B \ar[lu]_{j}\ar[ld]_{\beta}
 \\
& C    &
}
$$

Suppose that $A \sqcup B$ is graded by a group $G$. Since $i$ and $j$ map homogeneous elements to homogeneous elements, there exist $g,h, t \in G$ such that
$$i(a)\in (A \sqcup B)^{(g)};\text{ } i(b)\in (A \sqcup B)^{(h)}; \text{ } j(1_B), j(v)\in (A \sqcup B)^{(t)}.$$ Since
$a,c \in A^{(\xi)}$ and $b,d \in A^{\left(\xi^2\right)}$,
we have $i(c)\in (A \sqcup B)^{(g)}$ and
$i(d)\in (A \sqcup B)^{(h)}$.
Now (\ref{EqIcIdNe0}) implies $gh=hg$.
At the same time, since $\psi(j(v))=\beta(v)=y\ne 0$,
we have $j(v) \ne 0$
and $j(1_B)j(v)=j(v)$. Therefore $t^2=t$ and consequently $t=1_G$. Hence both $i(a)j(v)i(b)$
and $i(b)j(v)i(a)$ belong to $(A \sqcup B)^{(s)}$ where $s=g t h = gh = hg = htg$.
However $\psi(i(a)j(v)i(b))=xyz$ and $\psi(i(b)j(v)i(a))=zyx$ belong to different homogeneous components of $C$.
We get a contradiction. Hence $A \sqcup B$ does not exist.
\end{proof}

Below we calculate the equalizers.

\begin{proposition}\label{PropositionEqualizerGrAlg}
Let $\alpha, \beta \colon A \to B$ be homomorphisms of graded algebras.
Denote by $C$ the linear span of all homogeneous elements $a\in A$
such that $\alpha(a)=\beta(a)$. Let $i \colon C \to A$ be the corresponding
embedding. Then $i$ is an equalizer of $\alpha$ and $\beta$ both in
$\mathbf{GrAlg}_F$ and in $\mathbf{GrAlg}^1_F$.
(In the case of $\mathbf{GrAlg}^1_F$ we require that both $\alpha$ and $\beta$ are unital.)
\end{proposition}
\begin{proof}
 It is obvious that $C$ is a graded subalgebra of $A$.
We have to show that for every homomorphism $\gamma \colon D \to A$
of graded algebras such that $\alpha\gamma=\beta\gamma$ there exists a unique graded homomorphism
$\varphi \colon D \to C$ such that $\gamma = i\varphi$:

$\xymatrix{ C \ar[r]^i  & A \ar@<-0.5ex>[r]_\beta \ar@<0.5ex>[r]^\alpha & B \\
                       & D \ar[u]_\gamma \ar@{-->}[lu]^\varphi}$

However this follows from the fact that for every homogeneous $d\in D$
we have $\alpha(\gamma(d))=\beta(\gamma(d))$
which implies $\gamma(d)\in C$. Since $D$ is the linear span of its homogeneous components, $\varphi(D)\subseteq C$
and $\varphi$ is just the restriction of the codomain for $\gamma$.
\end{proof}

In the proposition below we give an example of a parallel pair of graded homomorphisms
having no coequalizer. Recall that in our notation $\mathcal F (x,y)$ is the free group of rank two generated by the symbols $x,y$.

\begin{proposition}
Let $G=\mathcal F (x,y)$ and let $B$ be the algebra over a field $F$ with the basis $1, a,b,c,d$
where the multiplication is defined by $\langle a,b,c,d \rangle_F^2=0$. Define a $G$-grading on $B$ as follows:
$$B^{(1)}=F1_B,\text{ } B^{(x)}=\langle a,c \rangle_F, \text{ }B^{(y)}=\langle b,d \rangle_F.$$
Let $A:=\langle 1, a, b\rangle_F \subset B$.
Denote by $\alpha \colon A \to B$ the natural embedding
and by $\beta \colon A \to B$ the unital homomorphism defined by $\beta(a)=b$
and $\beta(b)=a$.
Then the coequalizer of $\alpha$ and $\beta$ exists neither
in $\mathbf{GrAlg}_F$ nor in $\mathbf{GrAlg}^1_F$.
\end{proposition}
\begin{proof}
Suppose $\gamma \colon B \to C$ is the coequalizer of $\alpha$ and $\beta$.
Then \begin{equation}\label{EqNoCoeqInGrAlg}
\gamma(a)=\gamma(\alpha(a))=\gamma(\beta(a))=\gamma(b).\end{equation}

Let $D=\langle 1,a\rangle_F \subset B$ and let $\varphi \colon B \to D$ be the unital homomorphism
defined by
$$\varphi(a)=\varphi(b)=\varphi(c)=\varphi(d)=a.$$
Then there exists a unique morphism $\psi \colon C \to D$ such that $\psi\gamma = \varphi$:

$\xymatrix{ A \ar@<-0.5ex>[r]_\beta \ar@<0.5ex>[r]^\alpha & B  \ar[r]^\gamma \ar[d]^\varphi & C  \ar@{-->}[ld]^\psi \\
           &  D  }$

Hence $\psi(\gamma(a))=\varphi(a)=a\ne 0$ and $\gamma(a)\ne 0$.

Let $\theta \colon B \to A$ be a unital homomorphism defined by $\theta(a)=\theta(b)=0$,
$\theta(c)=a$, $\theta(d)=b$. Then there exists a unique morphism $\mu \colon C \to A$
such that $\mu\gamma = \theta$.
Since $\mu(\gamma(c))=a \ne 0$ and $\mu(\gamma(d))=b\ne 0$ belong to different homogenous components in $A$,
$\gamma(c)$ and $\gamma(d)$ belong to different homogenous components in $C$. Hence
$\gamma(a)$ and $\gamma(b)$ belong to different homogenous components in $C$
which is in contradiction with~(\ref{EqNoCoeqInGrAlg}) since we have shown above that $\gamma(a)\ne 0$.
\end{proof}

Now we describe monomorphisms and compare them with injective and graded injective homomorphisms.

\begin{proposition}\label{PropositionCriterionForMonomorphismGrAlg}
Let $f\colon A=\bigoplus_{g\in G} A^{(g)} \to B$ be a morphism in $\mathbf{GrAlg}_F$ or in $\mathbf{GrAlg}^1_F$. Then $f$ is a monomorphism if and only if $a \ne b$ for some $a,b\in \bigcup_{g\in G}
A^{(g)}$ always implies $f(a)\ne f(b)$. In particular, each monomorphism in $\mathbf{GrAlg}_F$ and in $\mathbf{GrAlg}^1_F$ is graded injective.
\end{proposition}
\begin{proof}
Let $C$ be the algebra of all polynomials in the variable $x$ with coefficients from $F$
with a constant term in the case of $\mathbf{GrAlg}^1_F$ and
without a constant term in the case of $\mathbf{GrAlg}_F$, endowed with the degree $\mathbb Z$-grading.
Then for every homogeneous $a\in A$ there exists a unique graded homomorphism $\lambda \colon C \to A$
such that $\lambda(x)=a$. Hence if there existed some $a,b\in \bigcup_{g\in G}
A^{(g)}$, $a \ne b$,
such that $f(a)=f(b)$, it would be possible to construct graded homomorphisms $\lambda, \mu \colon
C \to A$ such that $\lambda(x)=a$, $\mu(x)=b$, i.e. $\lambda\ne \mu$, but $f\lambda=f\mu$.
Therefore, the ``only if'' part is proved.

Suppose now that $f\colon A \to B$ is a graded homomorphism such that for every
homogeneous $a,b\in A$, $a \ne b$, we have $f(a)\ne f(b)$.
Let $\alpha,\beta \colon D \to A$ be two graded homomorphisms such that
$f\alpha = f\beta$. Then for every homogeneous $d\in D$
we have $\alpha(d),\beta(d)\in \bigcup_{g\in G} A^{(g)}$
and $f(\alpha(d))=f(\beta(d))$. Hence $\alpha(d)=\beta(d)$. Since $D$ is the linear span of its homogeneous
components, we have $\alpha = \beta$ and $f$ is a monomorphism.
\end{proof}

Note that not every graded injective homomorphism is a monomoprhism in $\mathbf{GrAlg}_F$ or $\mathbf{GrAlg}_F^1$.  For example, the augmentation map $\mathrm{aug} \colon FG \to F$, $\mathrm{aug}\left(\sum_{g\in G} \alpha_g u_g\right):=\sum_{g\in G} \alpha_g$,
is graded injective but not a monomorphism for a nontrivial group $G$.

Below we give an example of a monomorphism which is not an injective map.

\begin{example}
Let $A=F1_A\oplus Fa_1\oplus Fa_2 \oplus Fa_3$ where $\langle a_1, a_2, a_3 \rangle_F^2=0$.
Define a $\mathbb Z/4\mathbb Z$-grading on $A$ by $a_i \in A^{(\bar i)}$
for $i=1,2,3$. Let $B = A/(a_1+a_2+a_3)$ be endowed with the trivial grading.
Then the natural surjection $\pi \colon A \twoheadrightarrow B$
is a monomorphism both in $\mathbf{GrAlg}_F$ and $\mathbf{GrAlg}^1_F$.
\end{example}
\begin{proof}
Let $a,b \in \bigcup_{\bar i \in \mathbb Z/4\mathbb Z} A^{(\bar i)}$
such that $\pi(a)=\pi(b)$. Then $a-b = \alpha(a_1+a_2+a_3)$
for some $\alpha \in F$. Note that $\alpha$ cannot be nonzero, since otherwise
$a-b$ would have nonzero components in each of $A^{(\bar i)}$, $\bar i \ne \bar 0$.
Hence $a=b$ and Proposition~\ref{PropositionCriterionForMonomorphismGrAlg}
implies that $\pi$ is a monomorphism both in $\mathbf{GrAlg}_F$ and $\mathbf{GrAlg}^1_F$.
\end{proof}

\section{Categorical constructions in $\widetilde{\mathbf{GrAlg}_F}$
and $\widetilde{\mathbf{GrAlg}^1_F}$}\label{SectionLimColimTildeGrAlg}

We begin with an example of graded algebras for which the product in $\widetilde{\mathbf{GrAlg}_F}$
and $\widetilde{\mathbf{GrAlg}^1_F}$ exists and is quite different from their product as non-graded algebras.

Recall that by $(u_g)_{g\in G}$ we denote the standard basis in a group algebra $FG$.

\begin{proposition}\label{PropositionProductGroupAlgebrasTildeGrAlg} Let $G$ and $H$ be groups and let $F$ be a field.
Then $F(F^{\times} \times G \times H)$ is the product of $FG$ and $FH$ (with the standard gradings) in both categories $\widetilde{\mathbf{GrAlg}_F}$
and $\widetilde{\mathbf{GrAlg}^1_F}$.
\end{proposition}
\begin{proof}
Let $\pi_1 \colon F(F^{\times} \times G \times H) \to FG$
and $\pi_2 \colon F(F^{\times} \times G \times H) \to FH$
be the homomorphisms defined by $\pi_1(u_{(\alpha,g,h)})=\alpha u_g$
and $\pi_2(u_{(\alpha,g,h)})=u_h$
for $\alpha \in F^\times$, $g\in G$, $h\in H$.
Obviously, they are graded injective.

Suppose there exists a graded algebra $A$ and graded injective homomorphisms
$\varphi_1 \colon A \to FG$ and $\varphi_2 \colon A \to FH$.
We claim that there exists a unique graded injective homomorphism
$\varphi \colon A \to F(F^{\times} \times G \times H)$
such that the following diagram commutes:

$$\xymatrix{ & F(F^{\times} \times G \times H)
\ar[ld]_{\pi_1} \ar[rd]^{\pi_2} & \\
FG & & FH \\
& A \ar[lu]_{\varphi_1} \ar[ru]^{\varphi_2} \ar@{-->}[uu]^\varphi &
}
$$
First we notice that, since each graded component of $FG$ has dimension one
and $\varphi_1$ is graded injective,
each graded component of $A$ must have dimension at most one.
Suppose now that the graded injective homomorphism $\varphi$
indeed exists.
Let $a \in A$ be a nonzero homogeneous element.
Then $\varphi(a)$
must be homogeneous too, i.e. $\varphi(a)=\alpha u_{(\beta, g,h)}$
for some scalars $\alpha, \beta \in F^\times$ and group elements $g\in G$, $h\in H$.
Then $\varphi_1(a)=\pi_1\varphi(a)= \alpha\beta u_g$
and $\varphi_2(a)=\pi_2\varphi(a)= \alpha u_h$, i.e. $\varphi(a)$
is uniquely determined by $\varphi_1(a)$ and $\varphi_2(a)$.

Given $\varphi_1$ and $\varphi_2$, the homomorphism $\varphi$ is defined as follows.
If $\varphi_1(a)=\lambda u_g$ and $\varphi_2(a)=\mu u_h$,
then $\varphi(a)=\mu u_{(\lambda/\mu, g,h)}$.
\end{proof}
\begin{corollary}
If the field $F$ consists of more than $2$ elements, then the functors $R$ and $R_1$ do not have left adjoints.
\end{corollary}
\begin{proof}
Each functor that has a left adjoint preserves limits (see e.g.~\cite[Chapter~V, Section 5, Theorem~1]{MacLaneCatWork}) and, in particular, products. However, $$R(F(F^{\times} \times G \times H))=
R_1(F(F^{\times} \times G \times H))=F^{\times} \times G \times H
\ncong R(FG) \times R(FH)=G\times H$$
in the case when both groups $G$ and $H$ are finite.
\end{proof}
In the next section (see Propositions~\ref{PropositionAbsenceOfLeftAdjointUGGF} and~\ref{PropositionAbsenceOfLeftAdjointUGGF1}) we prove that the restriction on the cardinality of the field is superfluous, that is, over any field the functors $R$ and $R_1$ do not have left adjoints.

Now we present an example of algebras for which their product in
$\widetilde{\mathbf{GrAlg}_F}$ and $\widetilde{\mathbf{GrAlg}_F^1}$
does not exist.

\begin{theorem}
Let $A=F[a_1,a_2]$, $B=F[b_1,b_2,b_3]$, algebras of polynomials in commutative variables
with coefficients from a field $F$,
endowed with the degree $\mathbb Z$-grading.
Then the product of $A$ and $B$  exists neither in $\widetilde{\mathbf{GrAlg}_F}$ nor in $\widetilde{\mathbf{GrAlg}_F^1}$.
\end{theorem}
\begin{proof}
Suppose $A\times B$ is the product of $A$ and $B$
and $\pi \colon A\times B \to A$ and $\rho \colon A\times B \to B$ are the corresponding projections.
Denote by $C$ the algebra of all polynomials in the variable $x$ with coefficients from $F$
with a constant term in the case of $\mathbf{GrAlg}^1_F$ and
without a constant term in the case of $\mathbf{GrAlg}_F$, endowed with the degree $\mathbb Z$-grading.

Now we use the same trick as in the proof of Theorem~\ref{TheoremAbsensceOfProductsGrAlg}.
For any nonzero (and therefore non-nilpotent) homogeneous $a \in A$ and $b\in B$,
there exist unique graded injective homomorphisms $\alpha \colon C \to A$
and $\beta \colon C \to B$ such that $\alpha(x)= a$ and $\beta(x)=b$.
Now the definition of a product implies that there exists a unique graded injective
homomorphism $\psi \colon C \to A\times B$ such that $\pi(\psi(x))=a$
and $\rho(\psi(x))=b$:
$$\xymatrix{ & A\times B
\ar[ld]_{\pi} \ar[rd]^{\rho} & \\
A & & B \\
&  C \ar[lu]^\alpha \ar[ru]_\beta \ar@{-->}[uu]^\psi &
}
$$
 Since $\psi$ is uniquely determined by the image of $x$,
there exists the unique non-nilpotent homogeneous element $c:= \psi(x)\in A\times B$
such that $\pi(c)=a$, $\rho(c)=b$.

Mapping $x$ to the elements $a_i$, $b_k$, $a_i + a_j$, $b_k+b_\ell$, we obtain unique non-nilpotent homogeneous elements $c_{ik}, d_{ijk\ell} \in A \times B$, where $1\leqslant i,j \leqslant 2$, $1\leqslant k,\ell \leqslant 3$, $i\ne j$, $k\ne \ell$, such that
$$\pi(c_{ij})=a_i,\text{ } \pi(d_{ijk\ell})=a_i+a_j,\text{ } \rho(c_{ij})=b_j, \text{ }\rho(d_{ijk\ell})=b_k+b_\ell.$$
Since
$$\pi(c_{ik}+c_{j\ell})=\pi(d_{ijk\ell})=a_i+a_j,\text{ }
\rho(c_{ik}+c_{j\ell})=\rho(d_{ijk\ell})=b_k+b_\ell$$
and $a_i+a_j$ and $b_k+b_\ell$ are non-nilpotent,
$c_{ik}+c_{j\ell}$ is non-nilpotent too and we must have $c_{ik}+c_{j\ell} = d_{ijk\ell}$.
Since all the elements $c_{ik},c_{j\ell},d_{ijk\ell}$ are nonzero,
for each quadruple $(i,j,k,\ell)$ the elements $c_{ik},c_{j\ell},d_{ijk\ell}$ belong to the same homogeneous component of $A\times B$. Changing $i,j,k,\ell$, we obtain that
all the elements $c_{ik},d_{ijk\ell}$ for all values $(i,j,k,\ell)$
belong to the same homogeneous component of $A\times B$.
 Since $\pi(c_{11})=\pi(c_{12})=a_1$ and
 $\pi$ is graded injective, $c_{11}=c_{12}$.
 However $\rho(c_{11})=b_1\ne b_2=\rho(c_{12})$ and we get a contradiction.
\end{proof}

Below we show that coproducts in $\widetilde{\mathbf{GrAlg}_F}$
and $\widetilde{\mathbf{GrAlg}^1_F}$ do not always exist.

\begin{proposition} Let $G$ and $H$ be groups and let $F$ be a field.
Then the coproduct of $FG$ and $FH$ (with the standard gradings) in
the category $\widetilde{\mathbf{GrAlg}_F}$ does not exist.
\end{proposition}
\begin{proof}
Suppose $A$ is the coproduct of $FG$ and $FH$  in $\widetilde{\mathbf{GrAlg}_F}$ and $i_1 \colon FG \to A$
and $i_2 \colon FH \to A$ are the corresponding morphisms.
Let $\varphi_1 \colon FG \to F(G\times H)$ and
$\varphi_2 \colon FH \to F(G\times H)$ be the natural embeddings.
Then there must exist $\varphi \colon A \to F(G\times H)$
such that the following diagram commutes:
$$\xymatrix{ & A
\ar@{-->}[dd]^\varphi & \\
FG \ar[ru]^{i_1}\ar[rd]^{\varphi_1} & & FH \ar[lu]_{i_2}\ar[ld]_{\varphi_2} \\
& F(G \times H)    &
}
$$

In particular, $$\varphi(i_1(u_g)i_2(u_h))=u_{(g,1_H)}u_{(1_G,h)}= u_{(g,h)}\ne 0$$
and consequently $i_1(u_g)i_2(u_h)\ne 0$.

Let $FG \times FH := \lbrace (a, b) \mid a \in FG,\ b\in FH\rbrace$ be
the algebra with the componentwise operations
where the $G\times H$-grading is defined by $(u_g, 0)\in (FG \times FH)^{\bigl((g, 1_H )\bigr)}$
and $(0, u_h)\in (FG \times FH)^{\bigl((1_G, h)\bigr)}$ for all $g\in G$ and $h\in H$.

Now let $\psi_1 \colon FG \to FG \times FH$ and
$\psi_2 \colon FH \to FG\times FH$ be the natural embeddings.
Then there must exist $\psi \colon A \to FG \times FH$
such that the following diagram commutes:
$$\xymatrix{ & A
\ar@{-->}[dd]^\psi & \\
FG \ar[ru]^{i_1}\ar[rd]^{\psi_1} & & FH \ar[lu]_{i_2}\ar[ld]_{\psi_2} \\
& FG \times FH    &
}
$$
In particular, $$\psi(i_1(u_g)i_2(u_h))=(u_g, 0)(0, u_h)= 0$$ and we get a contradiction since $i_1(u_g)i_2(u_h)\ne 0$ is a homogeneous element as a product of homogeneous elements and $\psi$
is graded injective.
\end{proof}

At the same time, the coproduct of $FG$ and $FH$ in $\widetilde{\mathbf{GrAlg}_F^1}$ equals $F(G*H)$,
the group algebra of the coproduct of $G$ and $H$ in the category of groups.
In Proposition~\ref{PropositionCoproductDoesntExistTilde} below
we construct an example of graded algebras having no
coproduct in both $\widetilde{\mathbf{GrAlg}_F}$ and $\widetilde{\mathbf{GrAlg}_F^1}$.

\begin{proposition}\label{PropositionCoproductDoesntExistTilde} Let $F$ be a field
and let $A_i=\langle 1, a_i\rangle_F$, where $a_i^2=0$, $i=1,2$, be two two-dimensional algebras
 with the $\mathbb Z/2\mathbb Z$-grading defined by $a_i \in A_i^{(\bar 1)}$.
Then the coproduct of $A_1$ and $A_2$ exists neither in
the category
$\widetilde{\mathbf{GrAlg}_F}$ nor in the category $\widetilde{\mathbf{GrAlg}_F^1}$.
\end{proposition}
\begin{proof}
Suppose $A$ is the coproduct of $A_1$ and $A_2$ and $i_j \colon A_j \to A$, $j=1,2$, are the corresponding morphisms.
Let $A_0=\langle 1, a_1,a_2\rangle_F$ be the three-dimensional algebra
defined by $a_1^2=a_2^2=a_1a_2=a_2a_1 = 0$
with the $\mathbb Z/3\mathbb Z$-grading defined by $a_j \in A_0^{(\bar j)}$, $j=1,2$.
Let $\varphi_j \colon A_j \to A_0$ be the natural embeddings.

There must exist $\varphi \colon A \to A_0$
such that the following diagram commutes:
$$\xymatrix{ & A
\ar@{-->}[dd]^\varphi & \\
A_1 \ar[ru]^{i_1}\ar[rd]^{\varphi_1} & & A_2 \ar[lu]_{i_2}\ar[ld]_{\varphi_2}
 \\
& A_0    &
}
$$

In particular, $\varphi(i_1(a_1)i_2(a_2))=a_1 a_2=0$.
Since both $i_1(a_1)$ and $i_2(a_2)$
are homogeneous elements and $\varphi$ is graded injective, we get
\begin{equation}\label{eq:equal0}
i_1(a_1)i_2(a_2) = 0.
\end{equation}

Now let $B=\langle 1, b_1,b_2, b_1b_2\rangle_F$ be the four-dimensional algebra
defined by $b_1^2=b_2^2=b_2b_1 = 0$
and the $\mathbb Z/2\mathbb Z\times \mathbb Z/2\mathbb Z$-grading defined by $b_1 \in B^{(\bar 1,\bar 0)}$, $b_2 \in B^{(\bar 0,\bar 1)}$.
Let $\psi_j \colon A_j \to B$, where $j=1,2$, be the embeddings defined by $a_j \mapsto b_j$, $1\mapsto 1$.

There must exist $\psi \colon A \to B$
such that the following diagram commutes:
$$\xymatrix{ & A
\ar@{-->}[dd]^\psi & \\
A_1\ar[ru]^{i_1}\ar[rd]^{\psi_1} & & A_2 \ar[lu]_{i_2}\ar[ld]_{\psi_2} \\
& B & }$$
In particular, $\psi(i_1(a_1)i_2(a_2))=b_1 b_2 \ne 0$ and we get a contradiction with~\eqref{eq:equal0}.
\end{proof}

Note that the map $i$ in Proposition~\ref{PropositionEqualizerGrAlg} is injective and therefore graded injective.
Hence the embeddings $\widetilde{\mathbf{GrAlg}_F} \subset \mathbf{GrAlg}_F$ and $\widetilde{\mathbf{GrAlg}_F^1} \subset \mathbf{GrAlg}_F^1$ preserve equalizers.

\begin{proposition}
Let $f,g \colon A \to B$ be two different morphisms in $\widetilde{\mathbf{GrAlg}_F}$ or $\widetilde{\mathbf{GrAlg}_F^1}$. Suppose the grading on the algebra $B$ is trivial.
Then the coequalizer of $f$ and $g$ does not exist.
\end{proposition}
\begin{proof}
Suppose that there exists a coequalizer $h\colon B \to C$. Then $hf = hg$. Since $h$ is graded injective and the grading on $B$ is trivial, $h$ is injective and $f=g$. We get a contradiction.
\end{proof}

Now we show that monomorphisms in $\widetilde{\mathbf{GrAlg}_F}$ and $\widetilde{\mathbf{GrAlg}_F^1}$
admit the same description as in $\mathbf{GrAlg}_F$ and $\mathbf{GrAlg}_F^1$.

\begin{proposition}\label{PropositionCriterionForMonomorphismTildeGrAlg}
Let $f\colon A=\bigoplus_{g\in G} A^{(g)} \to B$ be a morphism in $\widetilde{\mathbf{GrAlg}_F}$ or in $\widetilde{\mathbf{GrAlg}_F^1}$. Then $f$ is a monomorphism if and only if $a \ne b$ for some $a,b\in \bigcup_{g\in G} A^{(g)}$ always implies $f(a)\ne f(b)$.
\end{proposition}
\begin{proof}
Suppose $f$ is a monomorphism.
Let $a,b\in \bigcup_{g\in G} A^{(g)}$ such that $f(a)= f(b)$.

Suppose first that $f(a)^k \ne 0$ for all $k\in\mathbb N$.
Since $f$ is graded injective and $a^i$ and $b^j$ are homogeneous,
we have $a^k\ne 0$, $b^k \ne 0$ for all $k\in\mathbb N$.
As before, denote by $C$ the algebra of all polynomials in the variable $x$ with coefficients from $F$ with a constant term in the case of $\mathbf{GrAlg}^1_F$ and
without a constant term in the case of $\mathbf{GrAlg}_F$, endowed with the degree $\mathbb Z$-grading.
Then there exist unique graded injective homomorphisms $\alpha,\beta \colon C \to A$
such that $\alpha(x)=a$, $\beta(x)=b$. We have $f\alpha=f\beta$
since $f(\alpha(x)) = f(a)=f(b)=f(\beta(x))$. Since $f$ is a monomorphism,
we have $\alpha = \beta$ and $a=\alpha(x)=\beta(x)= b$.

Consider now the case when $f(a)$ is nilpotent.
Define $k\in \mathbb N$ by $f(a)^k=0$, $f(a)^{k-1}\ne 0$. (If $f(a)=0$, we put $k:=1$.)
Since $f$ is graded injective and $a^i$ and $b^j$ are homogeneous,
we have $a^k=b^k=0$, $a^{k-1}\ne 0$, $b^{k-1}\ne 0$.
If $k=1$, we have $a=b=0$ and the ``only if'' part is proved.
Suppose $k > 1$. Let $\bar C = C/(x^k)$. Denote by $\bar x$ the image of $x$ in $\bar C$.
Then there exist unique graded injective homomorphisms $\alpha,\beta \colon \bar C \to A$
such that $\alpha(\bar x)=a$, $\beta(\bar x)=b$. We have $f\alpha=f\beta$
since $f(\alpha(\bar x)) = f(a)=f(b)=f(\beta(\bar x))$. Since $f$ is a monomorphism,
we have $\alpha = \beta$ and $a=\alpha(\bar x)=\beta(\bar x)= b$. Now the ``only if'' part is proved completely.

The ``if'' part follows from Proposition~\ref{PropositionCriterionForMonomorphismGrAlg},
since in this case $f$ is a monomorphism even in $\mathbf{GrAlg}_F$ (respectively, in $\mathbf{GrAlg}_F^1$).
\end{proof}

\section{Proof of Theorem~B}\label{SectionAbsenceGrAdjunctions}

In this section we show that, unlike functors considered in Section~\ref{SectionGradedAdjunctions}, the functors
$$R \colon \widetilde{\mathbf{GrAlg}_F} \to \mathbf{Grp}\text{ and }
R_1\colon \widetilde{\mathbf{GrAlg}_F^1}\to \mathbf{Grp}$$
defined in Section~\ref{SectionUniversalGradingGroupFunctors},
which assign to each grading its universal group,
have neither left nor right adjoints. The combination of the following three propositions implies Theorem B.
\begin{proposition}\label{PropositionAbsenceOfLeftAdjointUGGF}
The functor $R$ has no left adjoint.
\end{proposition}
\begin{proof} Suppose $K$ is left adjoint for $R$. Then
we have a natural bijection $$\widetilde{\mathbf{GrAlg}_F}(K(H), \Gamma) \to \mathbf{Grp}(H, R(\Gamma)).$$

We claim that for each group $H$ the grading $K(H)$ is a grading
on the zero algebra.
Suppose $K(H) \colon A=\bigoplus_{g\in G} A^{(g)}$ and $A\ne 0$.
Let $\Lambda$ be a set of indices such that $|\Lambda| > |\Hom(H,G_{K(H)})|$.
Consider the direct sum $\bigoplus_{\lambda \in \Lambda} A$
of copies of $A$ where each copy retains its $G$-grading $K(H)$. Denote by $\Xi$ the resulting $G$-grading on $\bigoplus_{\lambda \in \Lambda} A$.
Then $G_{\Xi} \cong G_{K(H)}$, however $$|\widetilde{\mathbf{GrAlg}_F}(K(H), \Xi)|
\geqslant |\Lambda| > |\Hom(H,G_{K(H)})|=|\mathbf{Grp}(H, R(\Xi))|$$
which contradicts the existence of the natural bijection. Hence
for each group $H$ the grading $K(H)$ is a grading on the zero algebra.
In particular, each set $\widetilde{\mathbf{GrAlg}_F}(K(H), \Gamma)$
contains exactly one element.

If $H$ is a nontrivial group
 and $FH$ is its group algebra with the standard grading $\Gamma$, then $\mathbf{Grp}(H, R(\Gamma))=\Hom(H,H)$
contains more than one element (at least the identity map and the homomorphism mapping everything to $1_H$) and we once again get a contradiction. Hence the left adjoint functor for
$R$ does not exist.
\end{proof}

The trick with infinite direct sums works only if the resulting algebra is allowed to be non-unital. On the other hand,
for unital algebras $A$ one can use the existence of the homomorphic embedding $F \cdot 1_A \to A$.

\begin{proposition}\label{PropositionAbsenceOfLeftAdjointUGGF1}
The functor $R_1$ has no left adjoint.
\end{proposition}
\begin{proof} Suppose $K$ is left adjoint for $R_1$. Then
we have a natural bijection $$\widetilde{\mathbf{GrAlg}^1_F}(K(H), \Gamma) \to \mathbf{Grp}(H, R_1(\Gamma)).$$

We claim that for any group $H$ the grading $K(H)$ is the grading
on an algebra isomorphic to $F$.
Indeed, let $H$ be a group and let $K(H) \colon A=\bigoplus_{g\in G} A^{(g)}$.
Denote by $\Upsilon$ the grading on $F$ by the trivial group. Then $\mathbf{Grp}(H, R_1(\Upsilon))$
and $\widetilde{\mathbf{GrAlg}^1_F}(K(H), \Upsilon)$ both
  consist of one element. In particular, there exists a unital graded injective homomorphism
  $\varphi \colon A \to F$. Hence there exists an ideal $\ker \varphi \subsetneqq A$ of codimension $1$.
  Now denote by $\Xi$ the grading on $A$ by the trivial group.
  If $\ker \varphi\ne 0$, then $\widetilde{\mathbf{GrAlg}_F^1}(K(H), \Xi)$
  consists of at least two different elements:
  \begin{enumerate}
    \item the identity map $A\to A$;
    \item the composition of the following two graded injective homomorphisms: $\varphi$
    and the embedding $F \cdot 1_A \to A$.
  \end{enumerate}
   Since $R_1(\Xi)$ is the trivial group and $\mathbf{Grp}(H, R_1(\Xi))$    consists of a single element, we get $\ker \varphi=0$ and $A \cong F$.

    In particular, $\widetilde{\mathbf{GrAlg}_F^1}(K(H), \Gamma)$ contains exactly one element for all $H$ and $\Gamma$.
  Considering  a nontrivial group $H$
   and its group algebra $FH$  with the standard grading $\Gamma$, we obtain that $\mathbf{Grp}(H, R(\Gamma))=\Hom(H,H)$ contains more than one element and we get a contradiction. Hence the left adjoint functor for $R_1$ does not exist.
\end{proof}

The proof of the absence of right adjoints is identical for both $R$ and $R_1$.

\begin{proposition}\label{PropositionAbsenceOfRightAdjointsUGGF}
The functors $R$ and $R_1$ have no right adjoints.
\end{proposition}
\begin{proof}
The proof is the same for both $R$ and $R_1$.
For definiteness consider the case of $R$.

Suppose $K$ is right adjoint for $R$. Then
we have a natural bijection $$\mathbf{Grp}(R(\Gamma), H) \to \widetilde{\mathbf{GrAlg}_F}(\Gamma, K(H)).$$

Note that the set $\mathbf{Grp}(R(\Gamma), H)$ is always
non-empty since it contains at least the homomorphism that maps everything to $1_H$. Fix a group $H$. Let $K(H) \colon A=\bigoplus_{g\in G} A^{(g)}$.
Let $B$ be an $F$-algebra with the cardinality $|B|$
that is greater than the cardinality $|A|$. For example,
$B = \End_F V$ where $V$ is a vector space with a basis that has a cardinality greater than $|A|$. Let $\Gamma$ be the grading on $B$ by the trivial group. Then for every $g\in G$ there exist no injective maps $B \to A^{(g)}$
and therefore the set $\widetilde{\mathbf{GrAlg}_F}(\Gamma, K(H))$
is empty. We get a contradiction. Therefore the right adjoint functor for
$R$ does not exist.
\end{proof}

The results above alongside with their proofs show that, in order to get indeed an adjuction, one must restrict the category of algebras
to algebras that are well determined by their universal grading group.

\section{Adjunction in the case of group algebras}\label{SectionFunctorGroupAlgebra}

Let $F$ be a field and let $\mathbf{Grp}'_F$ be the category where the objects
are groups $G$ that do not have non-trivial one dimensional representations (in other words, $H^1(G, F^{\times}) = 0$) and the morphisms are all group homomorphisms. Let $\mathbf{GrpAlg}'_F$ be the category
where the objects are group algebras $FG$ of groups $G$ from $\mathbf{Grp}'_F$
with the standard grading
and the morphisms are all non-zero graded homomorphisms.
Let $U$ be the functor $\mathbf{GrpAlg}'_F \to \mathbf{Grp}'_F$
defined by $U(FG)=G$ and $\varphi\left(FG_1^{(g)}\right)\subseteq FG_2^{\left(U(\varphi)(g) \right)}$
 for $\varphi \colon FG_1 \to FG_2$.
Denote by $F-$ the functor which associates to each group its group algebra over the field $F$.

\begin{proposition}\label{PropositionEquivGroupAlgGroups}
For every $G \in  \mathbf{Grp}'_F$
and $A\in \mathbf{GrpAlg}'_F$ there exists a bijection $\theta_{G, A} \colon \mathbf{GrpAlg}'_F(FG, A)
\to \mathbf{Grp}'_F(G, U(A))$ which is natural in $A$ and $G$. Furthermore, $FU(-)= 1_{\mathbf{GrpAlg}'_F}$
and $U(F-)= 1_{\mathbf{Grp}'_F}$, i.e. the categories
$\mathbf{Grp}'_F$ and $\mathbf{GrpAlg}'_F$ are isomorphic.
\end{proposition}
\begin{proof}
Let $\varphi\colon FG \to FH$ be a non-zero graded homomorphism.
Then $\varphi(u_{g_0}) \ne 0$ for some $g_0\in G$ and
$\varphi(u_{g_0})=\varphi(u_{g_0} u_g^{-1}) \varphi(u_g)\ne 0$
implies $\varphi(u_g)\ne 0$ for all $g\in G$. Hence $\varphi$ is graded injective. Therefore $\varphi$ is determined by group homomorphisms $\psi \colon G \to H$ and $\alpha \colon G \to F^\times$ such that $\varphi(u_g) = \alpha(g)u_{\psi(g)}$ for $g\in G$.
Since $G$ does not have non-trivial one dimensional representations,
$\alpha$ is trivial and we have the natural bijection $\theta_{G, A}$. The equalities $FU(-)= 1_{\mathbf{GrpAlg}'_F}$
and $U(F-)= 1_{\mathbf{Grp}'_F}$ are verified directly.
\end{proof}

{\bf Acknowledgements.}
We are grateful to the referee for his/her helpful and valuable comments and suggestions.


\begin{thebibliography}{99}

 \bibitem{AljaGia} Aljadeff, E., Giambruno, A. Multialternating
graded polynomials and growth of polynomial identities.
\textit{Proc. Amer. Math. Soc.} \textbf{141} (2013), 3055--3065.

 \bibitem{AljaGiaLa} Aljadeff, E., Giambruno, A., La~Mattina, D.
 Graded polynomial identities and exponential growth.
\textit{J. reine angew. Math.}, \textbf{650} (2011), 83--100.

\bibitem{BahtZaiGradedExp} Bahturin, Yu.\,A., Zaicev, M.\,V.
Identities of graded algebras and codimension growth.
{\itshape Trans. Amer. Math. Soc.}
\textbf{356}:10 (2004), 3939--3950.

\bibitem{BahturinZaicevSeghalGroupGrAssoc}
Bahturin, Yu.\,A., Zaicev, M.\,V., Sehgal, S.\,K.
Group gradings on associative algebras.
\textit{J. Algebra}, \textbf{241} (2001), 677--698.

\bibitem{BenabouBi} B\'enabou, J. Introduction to bicategories,
Reports of the Midwest category seminar, \textit{Lecture Notes in Mathematics}, \textbf{47} (1967), Springer, Berlin, 1--77.

\bibitem{ClaseJespersDelRio} Clase, M.\,V., Jespers, E., Del R\'\i o, \'A. Semigroup-graded rings with finite support. \textit{Glasgow Math. J.}, \textbf{38} (1996), 11--18.

\bibitem{ElduqueKochetov} Elduque, A., Kochetov M.\,V. Gradings on simple Lie algebras. \textit{AMS Mathematical Surveys and Monographs}. \textbf{189} (2013), 336 pp.

 \bibitem{GiaLa} Giambruno, A., La~Mattina, D.
 Graded polynomial identities and codimensions: computing the exponential growth.
 \textit{Adv. Math.}, \textbf{225} (2010), 859--881.

\bibitem{ginosargradings}
Ginosar, Y., Schnabel, O.
Groups of central type, maximal connected gradings and intrinsic
  fundamental groups of complex semisimple algebras.
(To appear in the Transactions of the AMS.) \url{https://doi.org/10.1090/tran/7457}

   \bibitem{ASGordienko3} Gordienko, A.\,S. Amitsur's conjecture for associative algebras
with a generalized Hopf action.
 \textit{J. Pure and Appl. Alg.},
\textbf{217}:8 (2013), 1395--1411.

\bibitem{ASGordienko9} Gordienko, A.\,S. Co-stability of radicals and its applications to PI-theory.
\textit{Algebra Colloqium}, \textbf{23}:3 (2016), 481--492.

\bibitem{GordienkoSchnabel}
Gordienko, A.\,S., Schnabel, O. On weak equivalences of gradings.
\textit{J. Algebra}, \textbf{501} (2018), 435--457.

\bibitem{KellyEnriched} Kelly, G.\,M. Basic Concepts of Enriched Category Theory. \textit{Lecture Notes in Mathematics} \textbf{64}, Cambridge University Press, 1982.

\bibitem{KellyTwoCat} Kelly, G.\,M., Street R.\,H. Review of the elements of 2-categories, Sydney Category Seminar, \textit{Lecture Notes in Mathematics}, \textbf{420} (1974), Springer, Berlin, 75--103.

\bibitem{MacLaneCatWork} Mac Lane, S. Categories for the working mathematician. \textit{Graduate texts in Mathematics} \textbf{5}, Springer-Verlag, New York, 1998.

\bibitem{Mazorchuk} Mazorchuk, V., Zhao, K. Graded simple Lie algebras and graded simple representations. (To appear in Manuscripta Mathematica.)
\url{https://doi.org/10.1007/s00229-017-0960-5}

\bibitem{NastasescuVanOyst} N\u ast\u asescu, C., Van Oystaeyen, F. Methods of graded rings.
 \textit{Lecture Notes in Mathematics} \textbf{1836}, Springer-Verlag, Berlin--Heidelberg, 2004.

\bibitem{PZ89} Patera, J., Zassenhaus, H. On Lie gradings. I, \textit{Linear Algebra Appl.}, \textbf{112}
(1989), 87--159.
\end{thebibliography}
\end{document}